\documentclass[final,twoside, a4paper]{amsart}
\usepackage[english]{babel}
\usepackage{amsmath, amsfonts,amssymb,amsopn,amscd,amsthm,mathrsfs}
\usepackage{bm}
\usepackage{graphicx}
 \usepackage{enumerate}
 \usepackage{array}
 \usepackage{hyperref}
 \hypersetup{colorlinks=true,linkcolor=blue,citecolor=blue,linktoc=page}
\usepackage{showkeys}
\numberwithin{equation}{section}
\usepackage{fancyhdr}
\newcommand{\eq}{\begin{equation}}
\newcommand{\qe}{\end{equation}}

\newcommand{\C}{\mathcal{C}}

\newcommand{\E}{\mathbb{E}}
\newcommand{\N}{\mathbb{N}}
\newcommand{\R}{\mathbb{R}}
\newcommand{\p}{\mathbb{P}}
\theoremstyle{plain}

\newtheorem{thm}{Theorem}
\newtheorem{lem}[thm]{Lemma}
\newtheorem{prop}[thm]{Proposition}
\newtheorem{cor}[thm]{Corollary}

\theoremstyle{definition}

\theoremstyle{remark}
\newtheorem*{rem}{Remark}

\usepackage{color}

\begin{document}
\sloppy
\pagestyle{headings} 
\title[Superconcentration of Gaussian stationary processes]{Some superconcentration inequalities for extrema of stationary Gaussian processes}
\date{Note of \today}
\author{Kevin Tanguy \\ University of Toulouse, France}
\address{Kevin Tanguy is with the Institute of Mathematics of Toulouse (CNRS UMR 5219). Universit\'e Paul Sabatier 31062 Toulouse, France.}
\email{ktanguy@math.univ-toulouse.fr}
\urladdr{http://perso.math.univ-toulouse.fr/ktanguy/}

\begin{abstract}
This note is concerned with concentration inequalities for extrema of stationary Gaussian processes.
It provides non-asymptotic tail inequalities which fully reflect the fluctuation rate, and as such improve upon standard Gaussian concentration. The arguments rely on the hypercontractive approach developed by Chatterjee for superconcentration variance bounds. Some statistical illustrations complete the exposition. 
\end{abstract}
\maketitle 
\section{Introduction}

\subsection{Convergence of extremes}

As an introduction, we recall some classical facts about
weak convergence of extrema of stationary Gaussian sequences and processes.

\subsubsection{Stationary Gaussian sequences}

Let $(X_i)_{i\geq 0}$ be a centered stationary Gaussian sequence such that
$\mathbb{E}\left[ X_i^2\right]=1$, $i\geq 0$, with covariance $\mathrm{Cov}(X_i,X_j)=\phi(|i-j|),\,i,j\geq 0$, where $\phi\,:\,\N\to\R$. An extensive study has been developed in this setting towards the asymptotic behaviour of the maximum $M_n=\max_{i=1,\ldots,n}X_i$ (Berman, Mittal, Pickands... cf. e.g.~\cite{Lead}). A sample result is the following theorem (see \cite{Lead}).

 \begin{thm}\label{CVd}
 Let $(X_i)_{i \geq 0}$ be a stationary Gaussian sequence with covariance
 function $\phi$ such that $\phi(n)\log n\to 0$ as $n\to\infty$. Then
 $$
 a_n(M_n-b_n)\underset{n\rightarrow\infty}{\longrightarrow} G
 $$
 in distribution where $a_n=(2\log n)^{1/2}$,
 $$
 b_n=(2\log n)^{1/2}-\frac{1}{2}(2\log n)^{-1/2}(\log\log n+\log 4\pi)
 $$
 and the random variable $G$ has a Gumbel distribution (for all $x\in\R$,
 $\p(G\leq x)=e^{-e^{-x}})$.
 \end{thm}
 
\subsubsection{Stationary Gaussian processes}

Let $(X_t)_{t\geq 0}$ be a centered stationary Gaussian
process such that $\E\left[X_t^2\right]=1,\, t\geq 0$, with covariance function $\mathrm{Cov}(X_s,X_t)=\phi\big(|t-s|\big)$, for $t,s\geq 0$, where  $\phi\,:\,[ 0,+\infty)\to\R$. Consider two behaviors, as $t\to 0$, of the covariance function $\phi$ :
 \begin{eqnarray}
 \phi(t)&=&1-\frac{\lambda_2t^2}{2}+o(t^2)\label{1.1}\\
  \phi(t)&=&1-C|t|^\alpha+o\big(|t|^{\alpha}\big)\label{1.2}
 \end{eqnarray}

The first case ensures that $X_t$ is differentiable and $\lambda_2=-\phi''(0)$ is a spectral moment, whereas the second case concerns non-differentiable (but continuous) processes (such as the Ornstein-Uhlenbeck process). For more details about this topic, see \cite{Lead}. For any $T>0$, set
$M_T=\sup_{t\in[ 0, T]} X_t$.

 \begin{thm}\label{Cvc}
 Let $(X_t)_{t\geq 0}$ be a stationary Gaussian process such that $\phi(t)\log t\to 0$ as $t\to\infty$. Then 
 \begin{equation*}
 a_T(M_T-b_T) \underset{T\to \infty}{\longrightarrow} G
 \end{equation*}
 in distribution where
$a_T=(2\log T)^{1/2}$, $b_T$ depends on the hypothesis \eqref{1.1} or \eqref{1.2} and $G$ has a Gumbel distribution.
  \end{thm}
  
The aim of this note is to quantify the preceding asymptotic statements into sharp concentration inequalities fully reflecting the fluctuation rate of the maximum (respectively the supremum). Such variance bounds with the correct scale were first obtained
in this context by Chatterjee in \cite{Chatt1},
as part of the superconcentration phenomenon. The results presented here
strengthen the variance bounds into exponential tail inequalities. One main result is the following statement.

\begin{thm}\label{GS}
Let $(X_i)_{i \geq 0}$ be a centered stationary Gaussian sequence with covariance function $\phi$. Assume that $\phi$ is non-increasing and satisfies $\phi(1)<1/2$.
Then, there exists $\alpha=\alpha(\phi)\in(0,1)$ and $c=c(\phi,\alpha) > 0$ such that
for all $n\geq 2$,
\begin{equation}\label{ICd}
\p\left(|M_n-\E[ M_n]|>t\right)\leq 6e^{-ct/\sqrt{\max(\phi(n^\alpha),1/\log n)}},\quad t\geq 0.
\end{equation}
\end{thm}

\begin{rem}
Under the hypothesis of Theorem \ref{CVd}, $\max\left(\phi(n^\alpha),1/\log n\right)=1/\log n$
for $n$ large enough,
which is exactly the fluctuation rate. Observe furthermore
that integrating \eqref {ICd} recovers the variance bounds of
\cite{Chatt1}.
\end{rem}

 It is important to compare Theorem~\ref {GS} to 
classical Gaussian concentration (see e.g.~\cite{Led}) which typically produces
\begin{equation}\label{TBSI}
\p\left( |M_n-\mathbb{E}\left[ M_n\right]|\geq t\right)\leq 2e^{-t^2/2},\quad t\geq0.
\end{equation}
While of Gaussian tail, such bounds do not reflect the fluctuations of the extremum $M_n$ of Theorem \ref{CVd}.
Moreover, with respect to this Gaussian bound,
Theorem \ref{GS} actually provides the correct tail behavior of the maximum $M_n$ in the form of a superconcentration inequality, in accordance with the fluctuation result and the limiting Gumbel distribution
 (since $ \mathbb{P} (G > x) \sim e^{-x}$ as $ x \to \infty$).
Let us also emphasizes that Theorem \ref{GS} covers the classical independent case, when all the $X_i$'s are independent, by taking $\phi=0$ and provides the following concentration inequality which will be useful in the last section,
\begin{equation}\label{Indc}
\p\left(\left|M_n-\E\left[M_n\right]\right|\geq t\right)\leq 6e^{-ct\sqrt{\log n}},\quad t\geq 0.
\end{equation}

Observe in addition that Theorem~\ref {GS} expresses a concentration
property of the maximum around its mean whereas, in the regime of the convergence of extremes
of Theorem~\ref {CVd},
the centerings are produced by explicit values $b_n$. Actually, up to numerical constants,
the same inequalities hold true around the constant $b_n$ instead of the mean.
To this task, it is enough to prove that
$\sup_n\E\left[ |a_n(M_n-b_n)| \right]<\infty$. Set $Z_n=a_n(M_n-b_n)$.
Letting $M_n'$ be an independent copy of $M_n$, set similarly $Z_n'=a_n(M_n'-b_n)$.
Now, integrating \eqref {ICd},
$\sup_n\E\left[ a_n|M_n-\E[ M_n]|\right]<\infty$. Hence $\sup_n\E\left[ |Z_n-Z_n'|\right]<\infty$ from which it easily follows that $\sup_{n}\E\left[ |Z_n|\right]<\infty.$ 

\medskip

The next statement is the analogue of Theorem~\ref {GS} for stationary
Gaussian processes, suitably quantifying Theorem~\ref{Cvc}. It is presented 
in the more general context of a centered stationary Gaussian field
$(X_t)_{t\in\R^d}$ such that
$\E\left[ X_t^2\right]=1$, $ t \in \R^d$,
with covariance function $\phi$.
According to \cite{Chatt1}, very
little seems actually to be known on the asymptotic fluctuations of the supremum of stationary Gaussian processes 
indexed by $\R^d$ when the dimension $d$ is greater than two. 
Some recent specific results are available for Gaussian fields with logarithmic correlation
(see e.g.~\cite{Aco,DRSV1,Mad,ZD} and \cite{DRSV2} for an overview on the subject).
However, extending similarly the variance bounds of \cite{Chatt1}, we obtain a concentration inequality for the supremum of a Gaussian field over a subset $A$ of $\R^d$.
The case $d=1$ thus covers Theorem~\ref{Cvc}.

\begin{thm}\label{GP}
Let $(X_t)_{t\in\R^d}$ be a stationary Gaussian Euclidean field with
covariance function $\phi$. Assume that $t\mapsto\phi(t)$ is non-increasing and $\phi(1)<1/2$. 
If $A$ is a subset of $\R^d$, denote by $N(A)=N(A,1)$
the minimal number of balls with radius $1$ needed to cover $A$. Set $M(A) = \sup_{s\in A}X_s$.
 Then, there exist $ C =C(\phi,d) > 0$ and $c=c(\phi,d) > 0$ only depending
on $\phi$ and $d$ such that,
for all $A \subset \R^d$ with $N(A)>1$,
$$
\p\left(|M_A-\E\left[ M_A\right]|\geq t\right)
\leq 6e^{-ct/\sqrt{K_{N(A)}}},\quad t\geq 0,
$$
where
$$
K_{N(A)}=\max\left(\phi\left(N(A)^C\right),1/\log N(A)\right).
$$
\end{thm}
\begin{rem}
When $d=1$ and $A=[ 0,T]$ for some $T>0$, $N(A)= T/2$ and
under the hypotheses of Theorem \ref{Cvc}, $K_{N(A)}=1/\log T$ for $T$ large enough.
\end{rem} 

 The proofs of Theorem \ref{GS} and \ref{GP} are based on the
hypercontractive approach developed by Chatterjee \cite{Chatt1} towards variance bounds.
The task will be to adapt the argument to reach exponential tail inequalities at the correct fluctuation regime.
This will be achieved via the corresponding variance bound at the level of Laplace transforms.

The paper is organized as follows. In the next section, we present the semigroup tools used to prove the main results. The third section is devoted to the proof of Theorems \ref{GS} and \ref{GP}.
In Section 4, we present some illustrations to further Gaussian models. In the final section, we present an illustration in statistical testing.

 
 \section{Framework and tools}
 
  \subsection{Notation}
  
 For any $n\geq 1$, denote by $X=(X_1,\ldots, X_n)$ a centered Gaussian vector, with maximum
$M_n=\max_{i=1,\ldots,n}X_i$. Set $I=\mathrm{argmax}_{i=1,\ldots,n}X_i$. 

The main technical part of this work is provided by the following extension of Chatterjee's approach
\cite{Chatt1} adapted to exponential concentration bounds. The very basis of the argument is hypercontractivity of the Ornstein-Uhlenbeck semigroup
together with the exponential version of Poincar\'e's inequality
(cf.~\cite {Bak,BGL,Chatt1}).

 \begin{thm}\label{THm}
Let $X=(X_1,\ldots,X_n)$ be a centered Gaussian vector with covariance matrix $\Gamma$. 
 Assume that for some $r_0\geq0$, there exists a non trivial covering $\C(r_0)$ of $\{1,\ldots ,n\}$ verifying the following properties:
\begin{itemize}
\item  for all $i,j\in\{1,\ldots,n\}$ such that $\Gamma_{ij}\geq r_0$, there exists
$D\in\C(r_0)$ such that $i,j\in D$;\\
\item there exists $C\geq1$ such that, $a.s.$, $\sum_{D\in\C(r_0)} 1_{\{ I\in D \} }\leq C$.\\
\end{itemize} 
\noindent Let $\rho(r_0)=\max_{D\in\C(r_0)}\p(I\in D)$. 

Then, for every $\theta\in\R$, 
\begin{equation}\label{SCI}
\mathrm{Var}\left(e^{\theta M_n/2}\right)\leq C \, \frac{\theta^2}{4}\left(r_0+\frac{1}{\log\left(1/\rho(r_0)\right)}\right)
\E\left[e^{\theta M_n}\right].
\end{equation}
In particular,
$$
\p(|M_n-\E[ M_n]|\geq t)\leq 6e^{-ct/\sqrt{K_{r_0}}},\quad t\geq 0,
$$
where $K_{r_0}=\max\big(r_0,\frac{1}{\log(1/\rho(r_0))}\big)$ and $c>0$ is a universal constant.
\end{thm}

\begin{rem}\label{Cov}
By monotone convergence, we can obtain the same result for a supremum instead of a maximum. This fact will be useful to obtain an application of Theorem~\ref{THm} for a Gaussian Euclidean field.\\
\end{rem}

\begin{proof}
As announced, the scheme of proof follows \cite {Chatt1}.
The starting point is the representation formula
\begin{equation}\label{OH}
\mathrm{Var}(f)=\int_0^{\infty}e^{-t}\E\left[\nabla f\cdot P_t\nabla f\right] dt
\end{equation}
for the variance of a (smooth) function $ f : \R^n \to \R$ along
the Ornstein-Uhlenbeck semigroup ${(P_t)}_{t \geq 0}$.
General references on the Ornstein-Uhlenbeck semigroup and more general Markov semigroups,
as well as such semigroup interpolation formulas are \cite{Bak,BGL}.

Following \cite{Chatt1},
given a centered Gaussian vector $X$ with covariance matrix $\Gamma =(\Gamma_{ij})_{1 \leq i,j \leq n}$,
apply \eqref {OH} to (a smooth approximation of) $ f = e^{\theta M/2}$ where
$M(x)=\max_{i=1,\ldots,n}(Bx)_i$ and $\Gamma = B \, {}^t \!B$. It yields that, for all $\theta\in\R$,
$$
\mathrm{Var}\left(e^{\theta M_n/2}\right)=\frac{\theta^2}{4}\int_0^{\infty} \! e^{-t} \, 
\E\bigg[\sum_{i,j=1}^n \Gamma_{ij} \, 
1_{ \{X\in A_i \} }  e^{\theta M_n/2}1_{ \{ X^t\in A_j \} }e^{\theta M_n^t/2}\bigg] dt .
$$
Here $\{X\in A_i\}=\{I=i\}$, $i = 1, \ldots, n$, and $X^t = e^{-t} X + \sqrt {1 - e^{-2t}} \, Y$,
where $Y$ is an independent copy of $X$, with corresponding maximum $M_n^t$. This process shares the same hypercontractivity property
as the Ornstein-Uhlenbeck process (cf.~\cite{Chatt1}). Indeed, assuming without loss of generality that $B$ is invertible (otherwise restrict to a subspace of $\R^n$), define the semigroup $(Q_t)_{t\geq 0}$ by 
$$
Q_tf(x)=\E\left[f\left(e^{-t}x+\sqrt{1-e^{-2t}}Y\right)\right]
$$
for $f\,:\,\R^n\to\R$ smooth enough. Let $Z$ be a standard Gaussian vector on $\R^n$
 so that $Y=BZ$ in law. Thus, by defining $g\,:\,\R^n\to\R$ as $g(x)=f(Bx)$, we have, for all $x\in\R^n$,
\begin{eqnarray*}
Q_tf(x)&=&\E\left[f\left(e^{-t}x+\sqrt{1-e^{-2t}}Y\right)\right]\\
&=&\E\left[f\left(e^{-t}BB^{-1}x+\sqrt{1-e^{-2t}}BZ\right)\right]\\
&=&\E\left[g\left(e^{-t}B^{-1}x+\sqrt{1-e^{-2t}}Z\right)\right]=P_tg(B^{-1}x)
\end{eqnarray*}
Therefore, by the hypercontractive property of $(P_t)_{t\geq 0}$ with the same exposants $p$ and $q$ as the Ornstein-Uhlenbeck semigroup, 
\begin{equation}\label{OU2}
\E\left[\left|Q_tf(X)\right|^q\right]^{1/q}=\|P_tg\|_q\leq\|g\|_p=\E\left[\left|f(X)\right|^p\right]^{1/p},
\end{equation}
where $\|\cdot\|_r$, for any $r\geq1$, stands for the $L^r$-norm with respect to the standard Gaussian measure on $\R^n$.\\
Now, for every $ t \geq 0$, with $I^t =\mathrm{argmax}_{i=1,\ldots,n}X^t_i$,
\begin{eqnarray*}
 \mathcal{I} &=& 
\E\bigg[\sum_{i,j=1}^n \Gamma_{ij} \, 1_{\{X\in A_i\}}e^{\theta M_n/2}
    1_{\{ X_t\in A_j \} }e^{\theta M_n^t/2}\bigg] \\
&=&\E\bigg[ e^{\theta(M_n+M_n^t)/2}\sum_{i,j=1}^n  \Gamma_{ij} 
   \, 1_{\{ I=i \} }1_{\{ I^t=j \} } \bigg ]\\
&=&\E\left[e^{\theta\big(M_n+M_n^t\big)/2} \, \Gamma_{II^t}\right]\\
&=&\sum_{k\geq 0}\E\left[ e^{\theta\big(M_n+M_n^t\big)/2}
   \, \Gamma_{II^t}1_{\{2^{-k-1}\leq \Gamma_{II^t}\leq2^{-k}\}}\right].
\end{eqnarray*}
To ease the notation, denote $\Gamma_{II^t}$ by $\Gamma$ and set
$F^t=M_n+M_n^t$. Let $k_0=\min\{k\geq 0\, ; r_0\leq 2^{-k-1}\}$.
Cutting the preceding sum into two parts, we obtain
\begin{eqnarray*}
\mathcal{I}&\leq& \sum_{k=0}^{k_0}2^{-k} \, \E\left[ e^{\theta F^t/2}1_{\{\Gamma\geq r_0\}}\right]+\sum_{k=k_0+1}^{\infty}2^{-k} \, \E\left[ e^{\theta F^t/2}1_{\{2^{-k-1}\leq \Gamma\leq 2^{-k} \}}\right]\\
&\leq& 2 \sum_{D\in\C(r_0)}\E\left[ e^{\theta F^t/2}1_{\{I,I^t\in D\}}\right]+\sum_{k\geq 0}r_0 \, \E\left[ e^{\theta F^t/2}1_{\{2^{-k-1}\leq \Gamma\leq 2^{-k}\}}\right]\\
&=& 2 \sum_{D\in\C(r_0)}\E\left[ e^{\theta F^t/2}1_{\{I,I^t\in D\}}\right]
    +r_0 \, \E\left[ e^{\theta F^t/2}\right] .
\end{eqnarray*}
By the Cauchy-Schwarz inequality and Gaussian rotational invariance,
$$
\E\left[ e^{\theta F^t/2}\right]
\leq \E\left[e^{\theta M_n}\right].
$$
On the other hand, by H\"older's inequality with $\frac {1}{p} + \frac {1}{q} = 1$,
\begin {equation*} \begin {split}
\E\left[ e^{\theta F^t/2}1_{\{I,I^t\in D\}} \right] 
 & = \E\left[ e^{\theta M_n/2}1_{\{I\in D\}} \, 
         e^{\theta M_n^t/2}1_{\{I^t\in D\}} \right] \\
 &\leq \E\left[ e^{\theta pM_n/2}1_{\{I\in D\}} \right]^{1/p} \,
        \E\left[ e^{\theta qM_n^t/2}1_{\{I^t\in D\}} \right]^{1/q}. \\
\end {split} \end {equation*}
Next, by the hypercontractivity property \eqref{OU2},
$$
\E\left[ e^{\theta qM_n^t/2}1_{\{I^t\in D\}} \right]^{1/q}
  \leq \E\left[ e^{\theta pM_n/2}1_{\{I\in D\}} \right]^{1/p}
$$
provided that $e^{2t} = \frac {q-1}{p-1}$, that is $p = 1 + e^{-t} < 2$.
As a consequence,
$$
\E\left[ e^{\theta F^t/2}1_{\{I,I^t\in D\}} \right] 
  \leq  \E\left[ e^{\theta pM_n/2}1_{\{I\in D\}} \right]^{2/p}
$$
and by a further use of H\"older's inequality,
$$
\E\left[ e^{\theta F^t/2}1_{\{I,I^t\in D\}} \right] 
\leq \p(I\in D)^{\frac{2-p}{p}} \, 
   \E\left[ e^{\theta M_n}1_{\{I\in D\}} \right] .
$$
Combining the preceding with the assumptions, we get that 
$$\mathcal{I}\leq 
   \left(r_0 +C \rho(r_0)^{\frac{2-p}{p}}\right)  \, \E\left[ e^{\theta M_n}\right].$$
Finally,
$$\mathrm{Var}\left(e^{\theta M_n/2}\right)\leq C \, \frac{\theta^2}{4}
 \left(r_0+ \int_0^{\infty}e^{-t}\rho(r_0)^{\tanh(t/2)}dt\right) \E\left[ e^{\theta M_n}\right].$$
The announced inequality \eqref {SCI} then follows since
$$
\int_0^{\infty}e^{-t}\left(\rho(r_0)\right)^{\tanh(t/2)}dt\leq\frac{1}{\log\left(1/\rho(r_0)\right)} \, .
$$

To end the proof of Theorem~\ref {THm} and produce exponential tail bounds, 
we combine \eqref{SCI} with the following
lemma (see \cite{Led}).
\begin{lem}\label{IPe}
Let $Z$ be a random variable and $K>0$. Assume that for every 
$ |\theta|\leq 2/\sqrt{K}$,
\begin{equation}\label{ISC}
\mathrm{Var}\left(e^{\theta Z/2}\right)\leq K \, \frac{\theta^2}{4} \, \E\left[e^{\theta Z}\right].
\end{equation}
Then,
$$
\p\left(|Z-\mathbb{E}\left[Z\right]|>t\right)\leq 6e^{-ct/\sqrt{K}},
$$
for all $t\geq 0$, with $c>0$ a numerical constant.
\end{lem}
\begin{rem}
 Such a lemma has been used in prior works \cite {BR,DHS}
in order to obtain exponential concentration tail bounds for first passage percolation models.
\end{rem}
\end{proof}

\section{Proofs of the main results}

\subsection{Proof of Theorem \ref{GS}}
 The task is to show that the hypothesis of Theorem~\ref {THm} are satisfied.
 For $0<\alpha<1$, choose $r_0=\phi(\lfloor n^\alpha\rfloor)$ and
$\C(r_0) = \{ D_1, D_2, \ldots \}$ where
$D_1=\{1,\ldots,2\lfloor n^{\alpha}\rfloor\}$, $D_2=\{\lfloor n^\alpha\rfloor,\ldots, 3\lfloor n^{\alpha}\rfloor\}$ and so on, where $\lfloor \cdot\rfloor$ denote the integer part of a real number.

It is easily seen that this covering satisfies the hypothesis of Theorem~\ref{THm}.
 Indeed, take $i,j\in\{1,\ldots,n\}$ such that $i<j$. By stationarity,
 $\mathrm{Cov}(X_{i+1},X_j)=\phi(j-i)$. So if $\phi(j-i)\geq r_0=\phi(\lfloor n^{\alpha}\rfloor)$, since $\phi$ is non-increasing, we must have $j-i\leq \lfloor n^{\alpha}\rfloor$. 
By construction, any index $i$ is at most in three different elements
$D$ of $\C(r_0)$.
So $i$ belongs to some $D_{i_1}$, $D_{i_2}$ and $D_{i_3}$ with $i_1<i_2<i_3$ and
$j$ belong to $D_{j_1}$, $D_{j_2}$ and $D_{j_3}$ with $j_1<j_2<j_3$. Since $j-i\leq n^{\alpha}$ and the length of any $D$ in $\C(r_0)$ is $2\lfloor n^\alpha\rfloor$, there exists $s\in\{1,2,3\}$ such that $D_{i_s}=D_{j_s}$
(draw a picture) and we can choose $D=D_{i_s} \in\C(r_0)$.

Clearly $\sum_{D\in\C(r_0)}1_{\{I\in D\}}\leq C$, where $C$ is a universal constant (say $C=3$).
To end the proof, we have to bound $\max_{D\in\C(r_0)}\p(I\in D)$. Let $D\in\C(r_0)$, and to ease the notation, consider $D=\{1,\ldots,2\lfloor n^{\alpha}\rfloor\}$ (the important feature is that $D$ contains $2\lfloor n^{\alpha}\rfloor$ elements). Then 
$$
\p(I\in D)=\sum_{i=1}^{2\lfloor n^{\alpha}\rfloor}\p(I=i).
$$
 By standard Gaussian concentration (see \cite{Led} or the appendix of \cite{Chatt1}), for all $i\in D$,
\begin{eqnarray*}
\p(I=i)&=&\p(X_i=M_n)\\
&\leq&\p(X_i\geq t)+\p(M_n\leq t)\\
&\leq& 2e^{-\E\left[ M_n\right]^2/2}.
\end{eqnarray*}
The final step of the argument is to achieve a lower bound on $\E\left[ M_n\right]$.
To this task, we make use of Sudakov's minoration (see \cite{Ad} or \cite{LT}).
Letting $i,j\in D$, $i\neq j$, 
\begin{eqnarray*}
\E\left[(X_i-X_j)^2\right]&=&2-2\phi\big(|j-i|\big)\geq2\big(1-\phi(1)\big)=\delta
\end{eqnarray*}
since $\phi\big(|j-i|\big)\leq \phi(1)$. By the assumption on $\phi(1)$, $\delta>0$,
and thus by Sudakov's minoration there exists $c>0$ (independent of $n$) such that
$\E\left[ M_n\right]\geq c\delta\sqrt{\log n}$.
As a consequence of the preceding,
$$
\p(I=i)\leq\frac{2}{n^{(c\delta)^2/2}}=\frac{2}{n^{\epsilon}} \, .
$$
Hence
$$
 \p(I\in D)\leq \frac{4}{n^{\epsilon-\alpha}}=\frac{4}{n^\eta} 
$$
where $\alpha$ is such that $\eta=\epsilon-\alpha>0$, namely $\alpha<(c\delta)^2/2<1$.

The hypotheses of Theorem~\ref {THm} are therefore fulfilled with $\rho(r_0)\leq 4/n^{\eta}$ concluding the proof
of Theorem~\ref{GS}.

\subsection{Proof of Theorem~\ref{GP}}

We again follow Chatterjee's proof of variance bounds for Gaussian Euclidean fields (see Theorem~9.12 in \cite{Chatt1}). For any Borel set $A\subset\mathbb{R}^d$, 
set $m(A)=\E\left[M(A)\right] = \E\left[\sup_{s \in A} X_s\right]$.
Under the hypothesis of the theorem, it is proved in \cite{Chatt1}
that for any $A\subset\mathbb{R}^d$ such that $N(A)>1$,
$$
c_1(\phi,d)\sqrt{\log N(A)}\leq m(A)\leq c_2(\phi,d)\sqrt{\log N(A)}
$$
where $c_1=c_1(\phi,d)$ and $c_2=c_2(\phi,d)$ are positive constants that
depend only on the covariance $\phi$ and the dimension $d$ (and not on the subset $A$).
Set then
$$
s_0=N(A)^{\frac{1}{8}(c_1/c_2)^2},
$$
and assume that $N(A)$ is large enough so that
$s_0>2$. Let $r_0=\phi(s_0)$. Take a  maximal $s_0$-net  of $A$ ( i.e. a set of points that are mutually separated from each other by a distance strictly greater than $s_0$ and is maximal with respect to this property), and let $\C(r_0)$ be the set of $2s_0$-ball around the points in the net. In fact,  $\C(r_0)$ is a covering of $A$ satisfying the condition of Theorem \ref{THm}. With this construction, it is easily seen that
$\sum_{D\in\C(r_0)}1_{\{I\in D\}}\leq C$ and
$$
\max_{D\in\C(r_0)}\mathbb{P}(I\in D)\leq \frac{1}{N(A)^{C(\phi,d)}} \, .
$$
The conclusion then directly follows from Theorem~\ref{THm} with $\rho(r_0)\leq 1/N(A)^{C(\phi,d)}$.

\section{Other Gaussian models}
In \cite{Chatt1}, Chatterjee exhibits different models  in order to illustrate the so-called superconcentration phenomenon. For each of them, he managed to reach a better variance bound than the one produced by standard Gaussian concentration. We present for some of these models an exponential version.
\begin{prop}\label{Ic}
For any $n\geq 2$, let $X=(X_1,\ldots, X_n)$ be a centered Gaussian vector with $\E\left[ X_i\right]^2=1$
and $\E\left[ X_iX_j\right]\leq\epsilon$ for all $i,j = 1, \ldots , n$ for some
$\epsilon >0$. Then\\
$$
\mathbb{P}(|M_n-\mathbb{E}\lbrack M_n\rbrack|\geq t)\leq 6e^{-ct/\sqrt{K_n}},\quad t\geq 0,
$$
where $K_n=\max(\epsilon,1/\log n)$.
\end{prop}

\begin{proof}
 Take $r_0>\epsilon$ and define $\C(r_0)=\left\{\{1\},\ldots,\{N\}\right\}$. 
It is easy to see that the covering $\C(r_0)$ enters the setting of Theorem~\ref{THm}.
Furthermore, $\sum_{D\in\C(r_0)}1_{\{I\in D\}}=1$ and $\max_{D\in\C(r_0)}\p(I\in D)\leq 1/n^{\eta}$ for some $\eta >0$. The conclusion follows. 
\end{proof}


 As another instance, consider the Gaussian field on $\{-1,1\}^n$, $n\geq1$,
defined in the following way. Let $X_1,\ldots, X_n$ be independent identically distributed standard normals, 
and define $f\,:\,\{-1,1\}^n\to\R$ as
$$
f(\sigma)=\sum_{i=1}^nX_i\sigma_i.
$$
\begin{cor}
There exists $S=(\sigma^1,\ldots,\sigma^N)\in\{-1,1\}^n$ for some $N$, such that, for all $\sigma\neq \sigma'\in S,$ $|\sigma\cdot\sigma'|\leq Cn^{2/3}$ and 
$$
\p\left(|f-\E\left[ f\right]|\geq t\right)\leq 6e^{-ct /n^{2/3}},
  \quad t \geq 0.
  $$
\end{cor}
\begin{proof}
Chatterjee proved in \cite{Chatt1} the existence of $S\in\{-1,1\}^n$ on which
${|\sigma\cdot\sigma'|\leq n^{2/3}}$. 
Since, $\mathrm{Cov}\big(f(\sigma),f(\sigma)\big)=\sigma\cdot\sigma'$, it suffices to apply
Proposition~\ref{Ic} to $f_{|S}$ and $\epsilon=n^{2/3}$.
\end{proof}

\section{Statistical testing}

 In this last section, we illustrate the preceding results on a statistical testing problem
analyzed in \cite{ABDL}. There, the authors
used standard Gaussian concentration to obtain an acceptance region for their test. Using some of the material developed here, the conclusion may be reinforced
in some instances in the form of a superconcentration bound.

\medskip
Following the framework and notation of \cite {ABDL}, we
observe an $n$ dimensional Gaussian vector $X=(X_1,\ldots,X_n)$ and raise the
following hypothesis problem:
\begin{itemize}
\item Under the null hypothesis $H_0$, the components of $X$ are
independent identically distributed standard normal random variables. Denote
then by $\p_0$ and $\E_0$ respectively the
underlying probability measure and expectation under $H_0$.\\
\item To describe the alternative hypothesis $H_1$, consider a class
$\C=\{S_1,\ldots,S_N\} $ of $N$ sets of indices such that $S_k\subset\{1,\ldots,n\}$ for all $k=1,\ldots, N$. Under $H_1$, there exists an $S\in\mathcal{C}$ such that
$$X_i\,\,\rm{has}\,\, \rm{distribution}
\left\{
\begin{array}{ll}
\mathcal{N}(0,1)\quad \mathrm{if}\, i\in S^c,\\
\mathcal{N}(\mu,1)\quad \mathrm{if}\, i\in S,
\end{array}
\right.
$$
where $\mu>0$ is a positive parameter. The components of $X$ are independent under $H_1$ as well. For any $S\in\C$, denote
then by $\p_S$ and $\E_S$ respectively the
underlying probability measure and expectation under $H_1$. We will also assume that every $S\in\C$ has the same cardinality $|S|=K$.
\end{itemize}

We recall that a test is a binary-valued function $f\,:\,\R^n\to\{0,1\}$. If $f(X)=0$, the
test accepts the null hypothesis, otherwise $H_0$ is rejected.  As in \cite{ABDL}, consider the risk of a test $f$ measured by 
$$R(f)=\p_0\big(f(X)=1)\big)+\frac{1}{N}\sum_{S\in\C}\p_S\big(f(X)=0)\big).$$
The authors of \cite{ABDL} are interested in determining, or at least estimating, the value of $\mu$ under which the risk can be made small. Among others results, they used a test based on maxima, called the scan-test,
 for which they showed that
$$f(X)=1 \, \iff \,   2\, \max_{S\in\C}X_S\geq 
    \mu K+\E_0\Big[\underset{S\in\C}{\max}X_S\Big] $$
where $X_S=\sum_{i\in S}X_i$ for some $S\in\C$.
They obtain the following result.
\begin{prop}\label{Lug}
The risk of the maximum test $f$ satisfies $R(f)\leq \delta$ whenever
$$\mu\geq\frac{1}{K} \, \E_0\Big[\underset{S\in\C}{\max}X_S\Big]
  +2\sqrt{\frac{2}{K}\log\frac{2}{\delta}}.$$
\end{prop}
Together with \eqref{Indc}, this bound may be improved
in accordance with the correct magnitude of the maximum of a standard Gaussian vector.
\begin{prop}\label{stat}
For any $n\geq 2$, the risk of the maximum test $f$ satisfies $R(f)\leq \delta$ whenever
$$\mu\geq\frac{1}{K} \, \E_0\Big[\underset{S\in\C}{\max}X_S\Big]
+\log\frac{6}{\delta}\times\frac{2}{c\sqrt{K\log N}} $$
where $c>0$ is a universal constant.
\end{prop}
\begin{proof}
We follow the proof of \cite{ABDL} to get that both, for every $t\geq 0$,
$$
\p_0\Big(\underset{S\in\C}{\max} X_S\geq \E_0\Big[\max_{S\in\C} X_S\Big]+t\Big)\leq 3e^{-ct\sqrt{\log (N)/K}}
$$
and, for any $S\in\C$,
\begin{eqnarray*}
\p_S\Big(\max_{S'\in\C} X_{S'}\leq\mu K-t\Big)&\leq&\p_{S}\Big(\max_{S'\in\C} X_{S'}\leq\E_{S'}\Big[\max_{S'\in\C}X_{S'}\Big]-t\Big)\\
&\leq& 3e^{-ct\sqrt{\log (N)/K}}
\end{eqnarray*}
since, under $H_1$, for a fixed $S\in\C$, $\mu K=\E_S\left[X_S\right]\leq\E_S\left[\max_{S'\in\C}X_{S'}\right]$ and $X_S$ has the same law as any $X_{S'}$. Set then $t$ such that
$$
2t = \mu K- \E\Big[\max_{S\in\C}X_S\Big]
$$
 which is positive according to the hypothesis. With the previous inequalities, we obtain a bound on the risk 
$$R(f)\leq 6e^{-ct\sqrt{\log(N)/K}}$$
 which may then be turned into $R(f)\leq \delta$ as in the statement.
\end{proof}
The new bound of Proposition~\ref {stat} is as good as 
the one of Proposition~\ref{Lug} when $\delta\simeq 1/N^{\alpha}$
for some $\alpha>0$. It is better than Proposition~\ref{Lug} when $\delta\simeq1/\log^\alpha(N)$ with
$\alpha>0$.  Nevertheless, Proposition \ref{Lug} is better than Proposition \ref{stat} when $\delta\simeq e^{-N^\alpha}$, for $\alpha>0$. However, it might not be relevant to consider such small risk.\\
\newline
\textit{Aknowledgment. I thank my Ph.D advisor M. Ledoux for introducing this problem to me and for fruitful discussions.}

\bigskip
\bibliography{Biblioa1.bib}
\bibliographystyle{plain}

\end{document}